\documentclass[final,1p,times]{elsarticle2}

 \usepackage{amssymb}

\usepackage{amsmath}
\usepackage{amsfonts}
\usepackage{latexsym}
\usepackage{mathrsfs}
\usepackage{times}
\usepackage{array}
\usepackage{amscd}
\usepackage{a4}
\usepackage{pstricks}

\usepackage{hyperref}
\usepackage{theorem}
\usepackage{graphicx}

\input{boxedeps} 
\SetEPSFDirectory{} 
\HideDisplacementBoxes
\SetRokickiEPSFSpecial  
\newtheorem{theorem}{Theorem}
\newenvironment{proof}{\paragraph{Proof:}}{\hfill$\square$}
\newtheorem{corollary}[theorem]{Corollary}
\newtheorem{lemma}[theorem]{Lemma}
\DeclareMathAlphabet{\ams}{U}{msb}{m}{n}

\journal{}

\begin{document}

\begin{frontmatter}



\title{All januarials constructed from Hecke groups}


\author[sm]{Saadia Mehwish}
\ead{saadiamehwish@gmail.com}
\author[qm]{Qaiser Mushtaq}
\ead{pir\_qmushtaq@yahoo.com}
\address[sm]{Department of Mathematics, The Islamia University of Bahawalpur, Bahawalpur, Pakistan}
\address[qm]{The Islamia University of Bahawalpur, Bahawalpur, Pakistan}

\date{Feb., 10, 2018}

\begin{abstract}
Professor Graham Higman defined januarial as a special instance of map
constructed from embedding of a coset diagram for an action of $\Delta
(2,\ell ,k)$, on finite sets yielding exactly two orbits of the product of
the two generators, having equal sizes. In this paper we determine a
condition for the existence of a januarial from $\Delta (2,\ell ,k),$ the
quotients of Hecke groups $H_{\Lambda _{\ell }},$ when acting on the
projective lines over finite fields $PL(F_{q})$. We develope a method to
find all the januarials from Hecke groups $H_{\Lambda _{\ell }}$, when the
triangle group $\Delta (2,\ell ,k)$ acts on $PL(F_{q})$. We evelove a
formula for calculating genus of coset diagram depending on the fixed
points. By using it, we determine genus of the januarials.
\end{abstract}




\end{frontmatter}

\section{Introduction}

It is well known that triangle group $\Delta (m,\ell ,k)$ has a presentation 
$\left\langle s,t:s^{m}=t^{l}=(st)^{k}=1\right\rangle $ (e.g. see \cite{mr}, 
\cite{ll} and \cite{qmfs}). For januarials, we need to fix $m=2$.

The widely studied Hecke groups $H_{\Lambda _{\ell }}$ admit a presentation 
\begin{equation*}
\left\langle a,b:a^{2}=(ab)^{\ell }=1\right\rangle
\end{equation*}%
where $a:z\rightarrow \frac{-1}{z}$, $b:z\rightarrow z+\Lambda _{\ell }$ and 
$\Lambda _{\ell }=2\cos (\pi /\ell ).$ In another interpretation $H_{\Lambda
_{\ell }}\cong C_{2}\otimes C_{\ell }$, that is, $H_{\Lambda _{\ell
}}=\Delta (2,\ell ,\infty ).$ Thus $\Delta (2,\ell ,k)$ are homomorphic
images of $H_{\Lambda _{\ell }}.$ Throughout this paper we consider triangle
groups as quotients of $H_{\Lambda _{\ell }}$.

Let $q$ be a prime-power $p^{r}$. Then the projective line over Galois field 
$F_{q}$, contains elements of $F_{q}$, together with $\infty $, and is
denoted by $PL(F_{q})$.

A coset diagram is a graphical representation of an action of any group with
finite presentation on a set. For instance, a coset diagram for the
permutation action of Hecke group on projective line has $\ell -gon$ for
each cycle of $y$ and an edge for each transposition of $x$.

Genus of a connected orientable surface is the number of handles on it.
Genus of a graph is the minimal integer $m$ such that the graph can be drawn
on the surface of genus $m$ without crossings \cite{at}.

Let $H_{\Lambda _{\ell }}\cong G=\left\langle x,y:x^{2}=y^{\ell
}=1\right\rangle $ where $x,y$ are linear fractional transformations $%
z\rightarrow \frac{az+b}{cz+d}$ where $a,b,c$ and $d\in 
\mathbb{Z}
$ and $ad-bc=1$. Let $\alpha $ be a non-degenerate homomorphism from $G$ to $%
PGL(2,q)$. The two non-degenerating homomorphisms $\alpha $ and $\beta $ are
said to be conjugates if there exist an inner automorphism $\rho $ such that 
$\alpha \rho =\beta .$ Let $\bar{x}=\alpha x$, $\bar{y}=\alpha y$ and $X,Y$
be the matrices representing $\bar{x}$ and $\bar{y}$ respectively. Then $%
\theta =(trXY)^{2}/\det XY$ is invariant of the conjugacy classes of $\alpha
.$

If $S_{1},S_{2},S_{3},\ldots ,S_{n}$ are finite sets then
inclusion-exclusion principle is symbolically expressed as 
\begin{eqnarray}
\left\vert S_{1}\cup S_{2}\cup S_{3}\cup \cdots \cup S_{n}\right\vert
&=&\sum_{i=1}^{n}\left\vert S_{i}\right\vert -\sum_{1\leq
i<j\leq n}\left\vert S_{i}\cap S_{j}\right\vert  \label{10} \\
&&+\sum_{1\leq i<j<l\leq n}\left\vert S_{i}\cap S_{j}\cap
S_{l}\right\vert -\cdots +  \notag \\
&&(-1)^{n+1}\left\vert S_{1}\cap S_{2}\cap S_{3}\cap \cdots \cap
S_{n}\right\vert .  \notag
\end{eqnarray}%
As every integer can be uniquely written as the product of primes so let $%
n=p_{1}^{r_{1}}p_{2}^{r_{2}}\cdots p_{s}^{r_{s}}.$ Using inclusion-exclusion
principle the Euler's phi function $\phi (n)$ is expressible as (see \cite%
{vl}) $\ $%
\begin{equation}
\phi (n)=n-\sum_{i=1}^{s}\frac{n}{p_{i}}+\sum_{1\leq i<j\leq
s}\frac{n}{p_{i}p_{j}}-\cdots +(-1)^{s}\frac{n}{p_{1}p_{2}\ldots p_{s}}.
\label{5}
\end{equation}

Professor Graham Higman FRS introduced januarials in his last lectures in
2001. He conceived the idea of januarials during his work on Hurwitz groups,
that are non trivial quotients of $\Delta (2,3,7)$ group, which he never
published. After his death a brief account of his lectures concerning
januarials was published \cite{gh} .

In \cite{gh} associates are used to find januarials from the modular group.
But his method does not tell, how many times one should repeat the process
of taking associates to ultimately get a januarial.

In this paper we prove when januarials exist and then use some pre-existing
tools to get an appropriate method of finding januarials. Further we show
all the distinct januarials constructible from the Hecke group $H_{\Lambda
_{\ell }}.$

\section{A condition for the existence of januarials}

In this section, we determine when januarials exist in the action of
triangle groups that are the quotients of Hecke group $H_{\Lambda _{\ell }}$
on $PL(F_{q})$.

\begin{theorem}
A januarial exists in the action of quotients of $H_{\Lambda _{\ell }},$ $%
\Delta (2,\ell ,k),$ on $PL(F_{q})$ if and only if $k=(q+1)/2,$ for all $q>3$
where $q=p^{r}$ and $p\neq 2$.
\end{theorem}

\begin{proof}
If%
\begin{equation*}
k=(q+1)/2\text{ then }(xy)^{(q+1)/2}=1
\end{equation*}

and the action of $xy$ on $PL(F_{q})$ gives the permutation $\rho _{xy}\in
S_{q+1}$ such that

\begin{equation*}
(\rho _{xy})^{(q+1)/2}=1.
\end{equation*}%
This means that the lengths $r_{i}$ of cycles of $\rho _{xy}$ are divisors
of $(q+1)/2.$ Here two cases arise: when $r_{i}\leq 2$ for all $i$ and when $%
2<r_{m}<k$ for any cycle.

Case I: If $r_{i}\leq 2$ then $(xy)^{2},$ which is linear fractional
transformation, fixes all the points of $PL(F_{q})$ since the only linear
fractional transformation which fixes more than two vertices is the identity
transformation \cite{qm} so $(xy)^{2}$ is trivial, that is $(xy)^{2}=1$. But
order of $xy$ which is $k$ for $q>3$ and $k>2,$ leads to a contradiction.

Case II If $2<r_{m}<k$ for any cycle then $(xy)^{r_{m}}$ fixes $r_{m}>2$
points of the cycle since the only linear fractional transformation which
fixes more than two vertices is the identity transformation \cite{qm} so $%
(xy)^{r_{m}}$ is trivial that is $(xy)^{r_{m}}=1$. But order of $xy$ which
is $k$ and $k>r_{m},$ is a contradiction.

So the only possibility is that $r_{i}=k$ for all $i.$ This implies that $%
\rho _{xy}$ has exactly two cycles of length $k$. That is, there are exactly
two orbits of $\left\langle xy\right\rangle $ of same size $k=(q+1)/2.$Hence
the result.

The converse follows from the definition of januarials.
\end{proof}

Since the two equal sized orbits in januarials are the result of action of
the cyclic group $\left\langle xy\right\rangle $ on $PL(F_{p})$ so the
question arrises, does $k=(q+1)/2$ assure the existence of two equal sized
orbits under the action of every cyclic group of order $k$ on $PL(F_{q})?$
We answer the question in the following theorem.

\begin{theorem}
The action of the cyclic subgroup $C_{k},$ of $H_{\Lambda _{\ell }},$ on $%
PL(F_{q})$ when $k=(q+1)/2$ give exactly two orbits of the same size $k,$
for all $q>3$ where $q=p^{r}$ and $p\neq 2$.
\end{theorem}

\begin{proof}
The proof is the same as that of Theorem 1.
\end{proof}

\section{Januarials from $PGL(2,q)$}

Next we use the above condition to find januarials through a new method. We
use the procedure described by F. Shaheen in \cite{fs1}. Let $X,Y$ denote
the elements in $GL(2,Z)$ which corresponds to the elements $x,y$ in $G$.
Then $X,Y$ will satisfy the relations

\begin{equation*}
X^{2}=Y^{\ell }=\lambda I
\end{equation*}

for some scalar $\lambda $. We choose $X,Y$ to be the matrices

\begin{equation*}
\left[ 
\begin{array}{cc}
a & ci \\ 
c & -a%
\end{array}%
\right] and\left[ 
\begin{array}{cc}
e & fi \\ 
f & b-e%
\end{array}%
\right]
\end{equation*}

respectively,where $i\neq 0$ , $a,c,e,f,i$ belong to $F_{q}$ and $b\equiv s(mod p)$ for some s in $F_{q}.$

Let $\Delta $ be the determinant of $X$. We assume that the determinant of $%
Y $ is $1$, so that we have

\begin{equation*}
-(a^{2}+ic^{2})=\Delta \neq 0
\end{equation*}

and 
\begin{equation*}
1+if^{2}+e^{2}-eb=0.
\end{equation*}

Let $r$ be the trace of $XY$, so that

\begin{equation*}
r=a(2e-b)+2icf.
\end{equation*}

Also%
\begin{equation*}
\det (XY)=\det X\det Y=\Delta .
\end{equation*}

We define a parameter $\theta $ as $r^{2}/\Delta $. For a pair $\bar{x},\bar{%
y}$ in $PGL(2,q)$, satisfying the relations

\begin{equation*}
\bar{x}^{2}=\bar{y}^{\ell }=1.
\end{equation*}

We denote by $D(\theta ,q,\ell )$, where $\theta $ belongs to $F_{q}$, the
coset diagram corresponding to the action of $G$ on $PL(F_{q})$.

Now by the condition in Theorem 1, we need to consider only those coset
diagrams in which every vertex is fixed by $(xy)^{k}.$ We consider the case
for $(XY)^{k}=\lambda I,$ for some scalar $\lambda .$

By the following equation, given by Q. Mushtaq and F. Shaheen in \cite{fs2};

\begin{eqnarray*}
(XY)^{k} &=&\{\binom{k-1}{o}r^{k-1}-\binom{k-2}{1}r^{k-3}\Delta +\binom{k-3}{%
2}r^{k-5}\Delta ^{2}-...\}XY \\
&&-\Delta \{\binom{k-2}{o}r^{k-2}-\binom{k-3}{1}r^{k-4}\Delta +\binom{k-4}{2}%
r^{k-6}\Delta ^{2}-...\}I
\end{eqnarray*}

we get $(XY)^{k}=\lambda I$, for some scalar $\lambda $, if and only if

\begin{equation*}
\binom{k-1}{o}r^{k-1}-\binom{k-2}{1}r^{k-3}\Delta +\binom{k-3}{2}%
r^{k-5}\Delta ^{2}-...=0.
\end{equation*}

Substituting $r^{2}=\theta \Delta $ in the above equation and simplifying it
we get an equation in $\theta $ say $g_{k}\left( \theta \right) .$ Let $%
\theta _{i^{\prime }s}$be the roots of $g_{k}\left( \theta \right) $. Now
using these $\theta _{i^{\prime }s}$ and by backward substitution we get $X$
and $Y,$ and ultimately $x,y\in PGL(2,Z).$

\begin{theorem}
The action of quotients of $H_{\Lambda _{\ell }},$ $\Delta (2,\ell ,k),$ on $%
PL(F_{p})$ for $k=(p+1)/2$, yield januarials for all those values of $\theta 
$ that are the roots of $g_{k}\left( \theta \right) $ excluding those of $%
g_{k/d}\left( \theta \right) $ for all $d\mid k.$
\end{theorem}

\begin{proof}
The roots of $g_{k/d}\left( \theta \right) $ which are also the roots of $%
g_{k}\left( \theta \right) $ satisfy $\left( xy\right) ^{k/d}=1.$ So there
are 2d orbits of $xy$ (by the arguments as in the proof of Theorem1) and
these values of $\theta $ therefore do not yield januarials. All other
values of $\theta $ yield januarial because they give exactly two orbits of $%
\left\langle xy\right\rangle $ of the same size.
\end{proof}

Now we find how many $\ell -januarials$ exist and determine all the
januarials that can be constructed from Hecke groups $H_{\Lambda _{\ell }}.$

\section{All $\ell $-Januarials from PGL(2,q)}

The following lemma is taken from \cite{gh}.

\begin{lemma}
\textit{The number of conjugacy classes of elements in }$PGL(2,q)$\textit{\
of order }$(q+1)/2$\textit{\ is }$1/2$\textit{\ }$\phi ((q+1)/2)$\textit{.
If }$z$\textit{\ is any element of order }$(q+1)/2$\textit{\ in }$PGL(2,q)$%
\textit{, then every one of these conjugacy classes intersects the subgroup
generated by }$z$\textit{\ in }$\{z^{i},z^{-i}\}$\textit{\ for exactly one }$%
i$\textit{\ coprime to }$(q+1)/2$\textit{.}
\end{lemma}

The next result was first proved by Q. Mushtaq in \cite{qm1} \ for $y$ of
order $3$ and then generalized for $y$ of order $\ell $ by F. Shaheen in 
\cite{fs1} as:

\begin{theorem}
Any element $g$ whose order is not equal to $1,2$ or $6,$ of $PGL(2,q)$ is
the image of $xy$ under some non-degenerate homomorphism from $G$ to $%
PGL(2,q)$.
\end{theorem}

From the above two results we can easily see that:

\begin{corollary}
\label{l3}For any prime $p>3$, the number of distinct $\ell $-januarials
constructible from $PGL(2,p)$ is $\frac{1}{2}\phi ((p+1)/2)$.
\end{corollary}

We also show that the number $\ell $-januarials constructible from $PGL(2,p)$
using the method described in Section 2 is $\frac{1}{2}(\phi (k))$.

\begin{theorem}
\label{list4}The number of roots, $N,$ of $g_{k}\left( \theta \right) $
excluding those of $g_{k/d}\left( \theta \right) ,$ for all $d\in 
\mathbb{Z}
$ such that $d\mid k,$ is given by 
\begin{equation}
N=\frac{1}{2}(\phi (k)).  \label{t}
\end{equation}
\end{theorem}

\begin{proof}
Let $N_{k}$ denote the set of roots of $g_{k}\left( \theta \right) .$ Since
every positive integer can be uniquely written in the form of product of
primes so let 
\begin{equation*}
k=p_{1}^{r_{1}}p_{2}^{r_{2}}\cdots p_{s}^{r_{s}}
\end{equation*}%
where $p_{i^{\prime }s}$ are primes such that $p_{1}^{r_{1}}<p_{2}^{r_{2}}<%
\cdots <p_{s}^{r_{s}}$ and $r_{i^{\prime }s}$ are integers.

$N_{k/p_{1}}\cup N_{k/p_{2}}\cup \cdots \cup N_{k/p_{s}}$ contain the roots
of $g_{k/d}\left( \theta \right) ,$ for all $d\in 
\mathbb{Z}
$ such that $d\mid k$, as $N_{k/d}\subset N_{k/p_{i}}$ for at least one $%
i\in \{1,2,\ldots ,s\}.$ So 
\begin{equation*}
N=\left\vert N_{k}\right\vert -\left\vert N_{k/p_{1}}\cup N_{k/p_{2}}\cup
\cup \cdots N_{k/p_{s}}\right\vert .
\end{equation*}%
Now by equation $\left( \ref{10}\right) $ 
\begin{eqnarray*}
\left\vert N_{k/p_{1}}\cup N_{k/p_{2}}\cup \cdots \cup
N_{k/p_{s}}\right\vert &=&\sum_{i=1}^{s}\left\vert
N_{k/p_{i}}\right\vert -\sum_{1\leq i<j\leq s}\left\vert
N_{k/p_{i}}\cap N_{k/p_{j}}\right\vert \\
&&+\sum_{1\leq i<j<l\leq s}\left\vert N_{k/p_{i}}\cap
N_{k/p_{j}}\cap N_{k/p_{l}}\right\vert \\
&&-\cdots +(-1)^{s+1}\left\vert N_{k/p_{1}}\cap N_{k/p_{2}}\cap \cdots \cap
N_{k/p_{s}}\right\vert
\end{eqnarray*}%
so 
\begin{eqnarray}
N &=&\left\vert N_{k}\right\vert -\sum_{i=1}^{s}\left\vert
N_{k/p_{i}}\right\vert +\sum_{1\leq i<j\leq s}\left\vert
N_{k/p_{i}}\cap N_{k/p_{j}}\right\vert  \label{t01} \\
&&-\sum_{1\leq i<j<l\leq s}\left\vert N_{k/p_{i}}\cap
N_{k/p_{j}}\cap N_{k/p_{l}}\right\vert +\cdots +  \notag \\
&&(-1)^{s}\left\vert N_{k/p_{1}}\cap N_{k/p_{2}}\cap \cdots \cap
N_{k/p_{s}}\right\vert .  \notag
\end{eqnarray}%
Now 
\begin{equation*}
N_{k/p_{i}}\cap N_{k/p_{j}}=N_{k/p_{i}p_{j}}.
\end{equation*}%
Similarly 
\begin{equation*}
N_{k/p_{1}}\cap N_{k/p_{2}}\cap \cdots \cap
N_{k/p_{s}}=N_{k/p_{1}p_{2}...p_{s}}
\end{equation*}%
and equation $\left( \ref{t01}\right) $ becomes%
\begin{eqnarray}
N &=&\left\vert N_{k}\right\vert -\sum_{i=1}^{s}\left\vert
N_{k/p_{i}}\right\vert +\sum_{1\leq i<j\leq s}\left\vert
N_{k/p_{i}p_{j}}\right\vert -  \label{t2} \\
&&\sum_{1\leq i<j<l\leq s}\left\vert
N_{k/p_{i}p_{j}p_{l}}\right\vert +\cdots +(-1)^{s}\left\vert
N_{k/p_{1}p_{2}\cdots p_{s}}\right\vert .  \notag
\end{eqnarray}%
The degree of $g_{k}\left( \theta \right) $ and hence $\left\vert
N_{k}\right\vert $ is $(k-1)/2$ when $k$ is odd and $k/2-1$ when $k$ is
even. So there are two cases.

Case I: When $k$ is odd. 
\begin{equation*}
\left\vert N_{k}\right\vert =(k-1)/2,
\end{equation*}%
\begin{equation*}
\left\vert N_{k/p_{i}}\right\vert =(k/p_{i}-1)/2,
\end{equation*}%
\begin{equation*}
\left\vert N_{k/p_{i}p_{j}}\right\vert =(k/p_{i}p_{j}-1)/2
\end{equation*}%
and%
\begin{equation*}
\left\vert N_{k/p_{1}}\cap N_{k/p_{2}}\cap \cdots \cap
N_{k/p_{s}}\right\vert =(k/p_{1}p_{2}...p_{s}-1)/2.
\end{equation*}%
Then equation $\left( \ref{t2}\right) $ becomes%
\begin{eqnarray*}
N &=&(k-1)/2-\sum_{i=1}^{s}(k/p_{i}-1)/2+\sum_{1\leq i<j\leq
s}(k/p_{i}p_{j}-1)/2-\sum_{1\leq i<j<l\leq s}(k/p_{i}p_{j}p_{l}-1)/2
\\
&&+\cdots +(-1)^{s}(k/p_{1}p_{2}\cdots p_{s}-1)/2
\end{eqnarray*}%
implying that%
\begin{equation}
N=\frac{1}{2}\left[ 
\begin{array}{c}
(k-1)-\sum_{i=1}^{s}(k/p_{i}-1)+\sum_{1\leq i<j\leq
s}(k/p_{i}p_{j}-1)-\sum_{1\leq i<j<l\leq s}(k/p_{i}p_{j}p_{l}-1) \\ 
+\cdots +(-1)^{s}(k/p_{1}p_{2}\cdots p_{s}-1)%
\end{array}%
\right] .  \label{t3}
\end{equation}%
Combinatorially it can be easily seen that%
\begin{equation*}
\left\vert \left\{ k/p_{i}:1\leq i\leq s\right\} \right\vert =\binom{s}{1},
\end{equation*}%
\begin{equation*}
\left\vert \left\{ k/p_{i}p_{j}:1\leq i<j\leq s\right\} \right\vert =\binom{s%
}{2}
\end{equation*}%
and%
\begin{equation*}
\left\vert \left\{ k/p_{i}p_{j}p_{l}:1\leq i<j<l\leq s\right\} \right\vert =%
\binom{s}{3}.
\end{equation*}%
Putting these values in equation $\left( \ref{t3}\right) $ we get%
\begin{equation*}
N=\frac{1}{2}\left[ 
\begin{array}{c}
k-\binom{s}{0}-\sum_{i=1}^{s}(k/p_{i})+\binom{s}{1}%
+\sum_{1\leq i<j\leq s}(k/p_{i}p_{j})-\binom{s}{2} \\ 
-\sum_{1\leq i<j<l\leq s}(k/p_{i}p_{j}p_{l})+\binom{s}{3}+\cdots
+(-1)^{s}\left( k/p_{1}p_{2}\cdots p_{s}-\binom{s}{s}\right)%
\end{array}%
\right]
\end{equation*}%
or%
\begin{equation}
N=\frac{1}{2}\left[ 
\begin{array}{c}
k-\sum_{i=1}^{s}k/p_{i}+\sum_{1\leq i<j\leq
s}(k/p_{i}p_{j})-\sum_{1\leq i<j<l\leq s}(k/p_{i}p_{j}p_{l}) \\ 
+\cdots +(-1)^{s}(k/p_{1}p_{2}\cdots p_{s}) \\ 
-\binom{s}{0}+\binom{s}{1}-\binom{s}{2}+\binom{s}{3}-\cdots +(-1)^{s+1}%
\binom{s}{s}%
\end{array}%
\right] .  \label{t4}
\end{equation}%
Since 
\begin{equation}
\binom{s}{0}-\binom{s}{1}+\binom{s}{2}-\binom{s}{3}+\cdots +(-1)^{s+1}\binom{%
s}{s}=0  \label{t5}
\end{equation}%
therefore using $\left( \ref{t5}\right) $ in equation $\left( \ref{t4}%
\right) $ we get%
\begin{equation*}
N=\frac{1}{2}\left[ 
\begin{array}{c}
k-\sum_{i=1}^{s}k/p_{i}+\sum_{1\leq i<j\leq
s}(k/p_{i}p_{j})-\sum_{1\leq i<j<l\leq s}(k/p_{i}p_{j}p_{l}) \\ 
+\cdots +(-1)^{s}(k/p_{1}p_{2}\cdots p_{s})%
\end{array}%
\right] .
\end{equation*}%
But by equation $\left( \ref{5}\right) $ \ 
\begin{equation}
\phi (k)=k-\sum_{i=1}^{s}\frac{k}{p_{i}}+\sum_{1\leq i<j\leq
s}\frac{k}{p_{i}p_{j}}-\cdots +(-1)^{s}\frac{k}{p_{1}p_{2}\cdots p_{s}}
\label{t6}
\end{equation}%
imply that 
\begin{equation*}
N=\frac{1}{2}\left( \phi (k)\right) .
\end{equation*}

Case II: When $k$ is even, we have $p_{1}=2.$ We have further two cases;

(a) When $r_{1}>1$ then all the $k/p_{i},k/p_{i}p_{j},\cdots
k/p_{1}p_{2}\cdots p_{s}$ are even so we have 
\begin{equation*}
\left\vert N_{k}\right\vert =(k/2-1),
\end{equation*}%
\begin{equation*}
\left\vert N_{k/p_{i}}\right\vert =(k/2p_{i}-1),
\end{equation*}%
\begin{equation*}
\left\vert N_{k/p_{i}p_{j}}\right\vert =(k/2p_{i}p_{j}-1)
\end{equation*}%
and%
\begin{equation*}
\left\vert N_{k/p_{1}}\cap N_{k/p_{2}}\cap \cdots \cap
N_{k/p_{s}}\right\vert =(k/2p_{1}p_{2}\cdots p_{s}-1)
\end{equation*}

(b) When $r_{1}=1$ we have $k/p_{1},k/p_{1}p_{j},\cdots k/p_{1}p_{2}\cdots
p_{s}$ odd whereas all other are even.

Simplifying equation $\left( \ref{t01}\right) $ for both cases II (a) and
(b) we get 
\begin{equation*}
N=\frac{1}{2}(\phi (k)).
\end{equation*}%
So the result is true for all the cases.
\end{proof}

In Corollary \ref{l3} and Theorem \ref{list4} it can be seen that all the
distinct $\ell $-januarials can be constructed from $PGL(2,p)$ using the
method described in Section 2. Hence this method of obtaining januarial is
precise as compared to the method of obtaining januarial using associates
which is described in \cite{gh}.

\section{Genus of the coset diagram for the action of $H_{\Lambda _{\ell }}$}

It is pertinent to find a formula for calculating genus (of the coset
diagram for the action of $H_{\Lambda _{\ell }}$ on $PL(F_{p})$) using \ the
fixed points in a coset diagram instead of edges and vertices, counting
which is time consuming for higher primes and increase the possibility of
errors.

Graham Higman's formula for the genus $g$ of a coset diagram is

\begin{equation*}
g=\frac{1}{2}\{2-(\nu -e+f)\}
\end{equation*}%
where g denotes the genus of the coset diagram, $\nu $ is the number of all
the orbits and cycles of $y$, $e$ is the number of nontrivial orbits or
cycles of $x$ and $f$ is the number of all the orbits or cycles of $xy$.

\begin{theorem}
Genus $g$ of the coset diagram for the action of $H_{\Lambda _{\ell
}}=\left\langle x,y:x^{2}=y^{\ell }=1\right\rangle $ on $PL(F_{p})$ is 
\begin{equation*}
g=1-\frac{1}{4k\ell }[\{2(k-\ell )-k\ell \}(p+1)+k\ell \{2(\eta _{y}+\eta
_{xy})+\eta _{x}\}-2(k\eta _{y}+\ell \eta _{xy})]
\end{equation*}%
where $k$ is the order of $xy$ in the coset diagram, $\eta _{x},\eta _{y}$
and $\eta _{xy}$ are the number of fixed points of $x,y$ and $xy$
respectively.
\end{theorem}

\begin{proof}
The genus $g$\ of a coset diagram can be calculated by using Graham Higman's
formula, that is 
\begin{equation*}
g=\frac{1}{2}[2-(v-e+f)]
\end{equation*}%
where $v$ is the number of orbits of $y$, $e$ is the number of non-trivial
orbits of $x$; $f$ is the number of orbits of $xy$.

Let 
\begin{equation*}
v=\mu _{y}+\eta _{y}
\end{equation*}%
where $\mu _{y}$ is number of non trivial cycles of $y$ and $\eta _{y}$ be
the number of trivial cycles or fixed points of $y$. As $y^{\ell }=1$ and $y$
makes a cycle of length $\ell $ so that values of $\mu _{y}$ and $\eta _{y}$
would depend on the quotient and the remainder respectively after dividing $%
p+1$ by $\ell $, the remainder would be $\eta _{y}.$ Then 
\begin{equation*}
\ell \mu _{y}=p+1-\eta _{y}
\end{equation*}

that is%
\begin{equation*}
v=\frac{p+1-\eta _{y}}{\ell }+\eta _{y}
\end{equation*}

or%
\begin{equation*}
v=\frac{p+1+(\ell -1)\eta _{y}}{\ell }.
\end{equation*}%
Similarly, $f$ is the number of all cycles of $xy$. We split $f$ as 
\begin{equation*}
f=\mu _{xy}+\eta _{xy}
\end{equation*}%
where $\mu _{xy}$ is the number of non-trivial cycles of $xy$ and $\eta
_{xy} $ is the number of trivial cycles. As $(xy)^{k}=1$ and $xy$ makes a
cycle of length $k$. 
\begin{equation*}
k\mu _{xy}=p+1-\eta _{xy}
\end{equation*}

or%
\begin{equation*}
f=\frac{p+1-\eta _{xy}}{k}+\eta _{xy}
\end{equation*}

or%
\begin{equation*}
f=\frac{p+1+(k-1)\eta _{xy}}{k}.
\end{equation*}%
Let $\eta _{x}$ denote the number of fixed points of $x$. Then 
\begin{equation*}
p+1=2e+\eta _{x}
\end{equation*}

or%
\begin{equation*}
e=\frac{p+1-\eta _{x}}{2}
\end{equation*}%
putting values of $v,f$, and $e$ in Graham Higman's formula to obtain the
required result. That is,

\begin{equation*}
g=\frac{1}{2}\{2-(\frac{p+1+(\ell -1)\eta _{y}}{\ell }-\frac{p+1-\eta _{x}}{2%
}+\frac{p+1+(k-1)\eta _{xy}}{k})\}
\end{equation*}

or%
\begin{equation*}
g=1-\frac{1}{4k\ell }[\{2(k+\ell )-k\ell \}(p+1)+k\ell \{2(\eta _{y}+\eta
_{xy})+\eta _{x}\}-2(k\eta _{y}+\ell \eta _{xy})].
\end{equation*}
\end{proof}

For januarials we have $k=\frac{p+1}{2}$ and $\eta _{xy}=0.$\ Hence the
following corollary is in order.

\begin{corollary}
Genus of a januarial is 
\begin{equation*}
g=-\frac{p+1-\eta _{y}}{2\ell }+\frac{1}{4}(p+1-2\eta _{y}-\eta _{x})
\end{equation*}%
where $\eta _{x}$ and $\eta _{y}$ are the fixed points of $x$ and $y$
respectively.
\end{corollary}

\section{Illustration}

\begin{figure}[h]
  \centering
\includegraphics[scale=0.75]{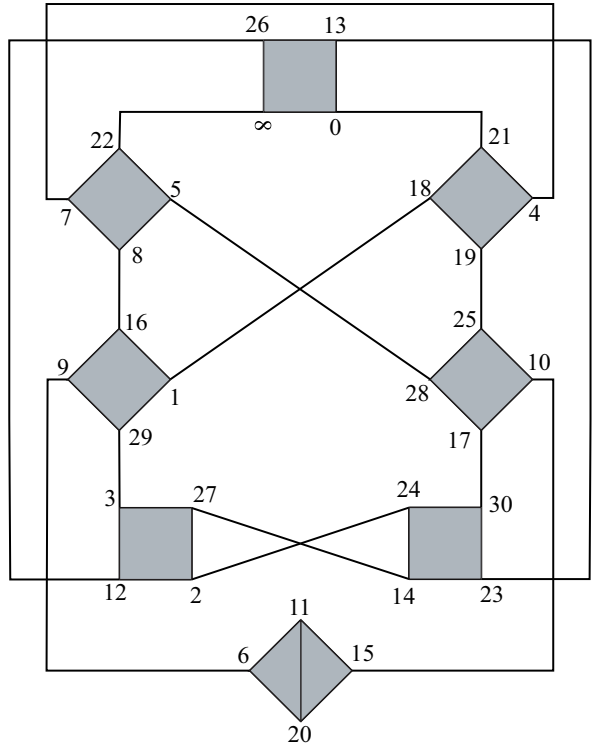}
  \caption{Coset diagram $D(7,31,4)$.}
\label{fig1}
\end{figure}

\begin{figure}[h]
  \centering
\includegraphics[scale=0.75]{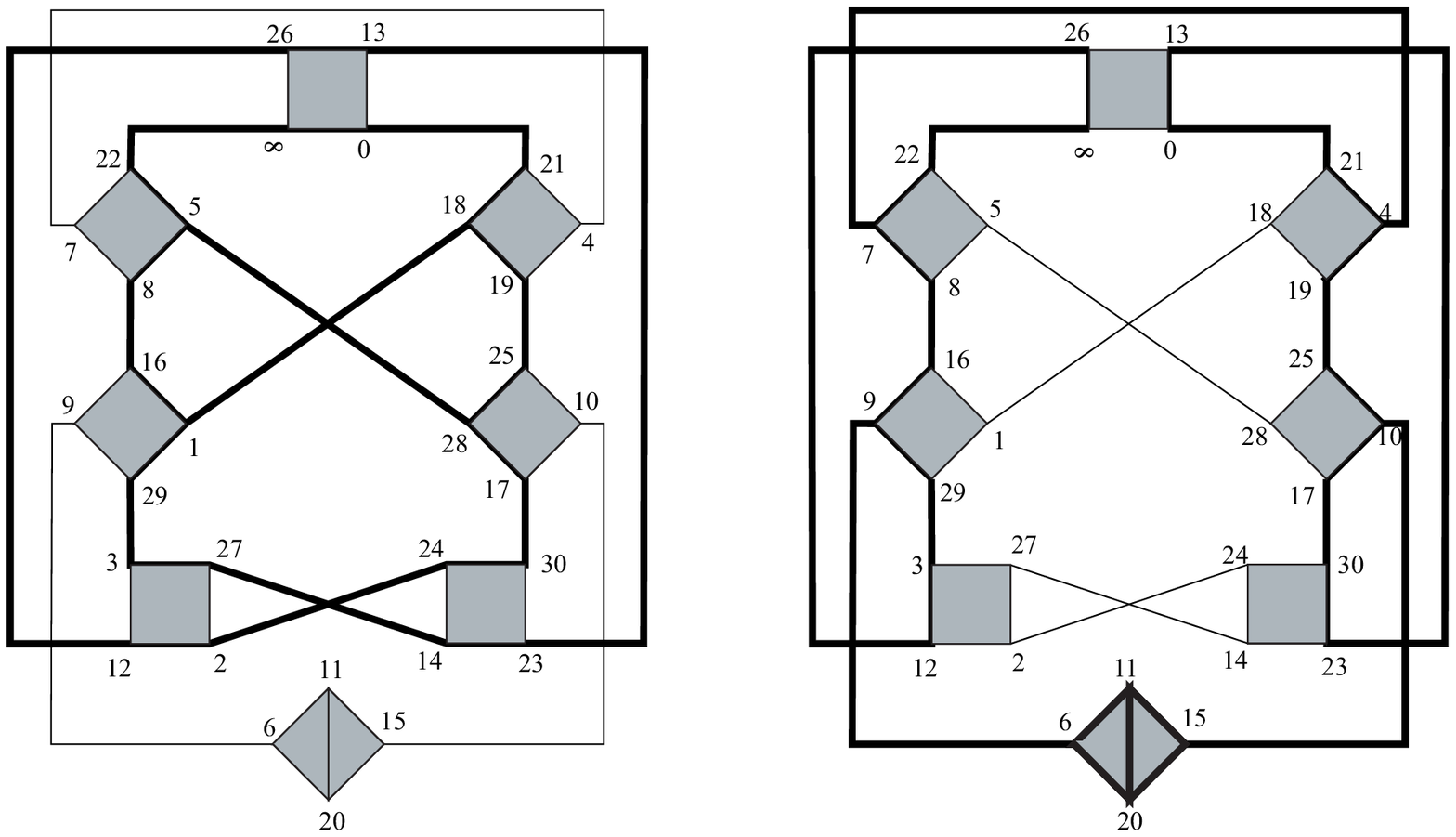}
  \caption{The two orbits of $\bar{x}\bar{y}$ highlighted with bold lines.}
\label{fig2}
\end{figure}

In the following we illustrate the above mentioned method for its better
understanding. If we take $p=31$, then by Theorem 1 we get $k=16$, and%
\begin{equation*}
g_{16}(\theta )=1x^{7}-14x^{6}+78x^{5}-220x^{4}+330x^{3}-252x^{2}+84x-8.
\end{equation*}

The zeros of $g_{16}(\theta )$ excluding those of $g_{8}(\theta
),g_{4}(\theta )$ and $g_{2}(\theta )$ are $7,16,19$ and $28$.

Now for a $4-januarial$ if we consider $\theta =7,$ we get%
\begin{eqnarray*}
x &:&z\rightarrow \frac{3z+30}{10z-3} \\
y &:&z\rightarrow \frac{42}{14z+8}.
\end{eqnarray*}%

The action of $x$ and $y$ on $PL(F_{31})$ give permutations:%
\begin{eqnarray*}
x%
\bar{}
&=&(0,21)(1,18)(2,24)(3,29)(4,7)(5,28)(6,9)(8,16)(10,15) \\
&&(11,20)(12,26)(13,23)(14,27)(17,30)(19,25)(22,\infty ) \\
y%
\bar{}
&=&(0,13,26,\infty )(1,16,9,29)(2,27,3,12)(4,21,18,19) \\
&&(5,22,7,8)(6,20,15,11)(10,25,28,17)(14,23,30,24) \\
x%
\bar{}%
y%
\bar{}
&=&(0,18,16,5,17,24,27,23,26,2,14,3,1,19,28,22) \\
&&(\infty ,7,21,13,30,10,11,15,25,4,8,9,20,6,29,12).
\end{eqnarray*}%
Hence the coset diagram $D(7,31,4)$ and the two orbits of $\bar{x}\bar{y}$ highlighted with bold lines are depicted in the Figures \ref{fig1} and \ref{fig2}.

According to Corollary 2, genus of this januarial is 4.






\end{document}